\documentclass[12pt]{article}

\usepackage{amsmath}
\usepackage{amssymb}
\usepackage{amsthm}
\usepackage{amsfonts}
\usepackage{amscd}
\usepackage{graphicx}
\usepackage[T1]{fontenc}    
\usepackage[usenames,dvipsnames]{color}
\usepackage{caption}
\usepackage[numbers,sort&compress]{natbib}
\usepackage{hyperref}
\usepackage{geometry}
\usepackage[medium]{titlesec}
\usepackage{etoolbox} 
\usepackage[capitalise]{cleveref} 
\usepackage{algorithm}
\usepackage{algpseudocode}

\newtheorem{lemma}{Lemma}
\newtheorem{proposition}[lemma]{Proposition}

\theoremstyle{definition}

\newtheorem{example}[lemma]{Example}
\newtheorem{remark}[lemma]{Remark}

\definecolor{colLinkBlue}{RGB}{23,111,192} 
\definecolor{colCiteGreen}{RGB}{8,144,8} 
\definecolor{colP}{RGB}{170,0,255} 
\definecolor{colLP}{RGB}{255,170,255} 
\definecolor{colBrown}{RGB}{156,99,49} 
\definecolor{colGray}{RGB}{128,128,128} 
\definecolor{colGreen}{RGB}{0,204,0} 
\definecolor{colO}{RGB}{255,170,0} 
\definecolor{colLB}{RGB}{143,189,211} 

\hypersetup{colorlinks,urlcolor=colLinkBlue,citecolor=colCiteGreen,linkcolor=colLinkBlue}
\titlespacing*{\section}{0pt}{0mm}{0mm}
\titlespacing*{\subsection}{0pt}{0mm}{0mm}
\titlespacing*{\paragraph}{0pt}{0mm}{0mm}
\newcommand{\myspace}{\setlength{\abovedisplayskip}{1mm}\setlength{\belowdisplayskip}{1mm}}
\algrenewcommand\algorithmicwhile{\textbf{While}}
\algrenewcommand\algorithmicfor{\textbf{For}}
\algrenewcommand\algorithmicdo{\textbf{Do}}
\algrenewcommand\algorithmicif{\textbf{If}}
\algrenewcommand\algorithmicthen{\textbf{Then}}
\algrenewcommand\algorithmicelse{\textbf{Else}}
\algrenewcommand\algorithmicend{\textbf{End}}
\algrenewcommand\algorithmicreturn{\textbf{Return}}

\newcommand{\ci}{\mathit{i}}
\renewcommand{\P}{{\mathbb{P}}}
\newcommand{\R}{{\mathbb{R}}}

\begin{document}
\myspace

\title{Calibrating Figures}
\author{Niels Lubbes, Josef Schicho}

\maketitle

\begin{abstract}
It is known that a camera can be calibrated using three pictures of either squares, or spheres,
or surfaces of revolution. We give a new method to calibrate a camera with the picture of a single torus.

\end{abstract}

\section*{Introduction}
\label{sec:0}

Central projection is a mathematical model of a camera that makes a 2D picture of a 3D scenario. In projective geometry,
a central projection is a map~$\pi$ from projective three-space minus a point -- the center $o$ -- to the projective plane. The points in $\P^2$
correspond to lines through $o$; any $q\ne o$ is mapped to the point corresponding to the unique line through $o$ and $q$.
A {\em calibrated camera} allows to determine the viewing angle at $o$ between two points $p,q$ in 3-space when their images $\pi(p),\pi(q)$ are given
(for instance, human sight can do that). The viewing angle is equal to the {\em elliptic distance} of the image points. A classical approach to
elliptic geometry is to derive elliptic distance from the {\em elliptic absolute}, a conic in $\P^2$ defined by a quadratic form with real
coefficients that does not have any real points. Hence a camera is calibrated as soon as the equation of the elliptic absolute is known.

The set of possible elliptic absolutes in $\P^2$ is five-dimensional. It is well-known that special figures in Euclidean space, when visible in the picture,
give equational constraints for the elliptic absolute. The simplest such figure is a square. The image of four vertices of a planar square give
rise to two linear equational constraints for the elliptic absolute. Consequently, the elliptic absolute can be determined uniquely from the pictures of three squares
(see \cite[Example~8.18]{HartleyZisserman}). Another well-known example is provided by a surface of revolution. Its picture has a reflection symmetry,
i.e., a projective transformation of order two preserving it. This symmetry also preserves the elliptic absolute, which
again gives rise to two linear constraints for the elliptic absolute (see \cite{WMC}). A special case is the sphere. Its picture is a conic,
and a plane conic has a two-dimensional set of projective reflection symmetries. Unfortunately, not all these reflections also preserve the elliptic
absolute. Still, it is possible to recover the elliptic absolute from the pictures of three spheres (see \cite{ZZW}).

The new result in this paper deals with another special case of a surface of revolution, namely the torus.
In \Cref{sec:2}, we show that the picture of a single torus determines finitely many candidates for the elliptic absolute, assuming a generic camera position.
\Cref{sec:1} explains a well-known relation between camera calibration and elliptic geometry which is fundamental for the remaining part.
\Cref{sec:2} explains calibration by surfaces of revolution; this is well-known, but we include this discussion because we also need it
for calibrations with a torus.

\section{Euclidean Absolute and Elliptic Absolute}
\label{sec:1}

In this paper, we work mainly with objects in the real projective plane $\P^2$ or real projective space $\P^3$.
At several places we also need to do constructions using non-real complex points.
In particular, we consider projective varieties that have no real points, despite being defined by forms with real coefficients.

Let $C_A\subset\P^3$ be an irreducible conic curve defined by a linear and a quadratic form with real coefficients, but without any real points.
In a suitable projective coordinate system, this curve is defined by $w$ and $x^2+y^2+z^2$. The subgroup of the projective linear group
consisting of all projective automorphisms that preserve $C_A$
also preserves angles and the ratio between lengths. In the sense of Klein's Erlangen program, specifying the group defines the angles and distances
between lengths. In its role as an object defining angles in $\P^3$, we call $C_A$ the {\em Euclidean absolute}, and the unique plane containing $C_A$
the {\em infinite plane}. In the above coordinates, the infinite plane is defined by $w$. Throughout the paper, we assume that there
is a fixed Euclidean absolute. This implies that we may speak of the angle between two given lines.

We consider the central projection $\pi:(\P^3\setminus\{o\})\to\P^2$ defined by a matrix in~$\R^{3\times 4}$. The one-dimensional
kernel of this map is the projective representation of the center $o$.
If the center $o$ is not lying in the infinite plane, then we may assume that $o=(0:0:0:1)$ and that $\pi$ is the map $(x:y:z:w)\mapsto (x:y:z)$.
If the Euclidean absolute is defined by $w$ and $x^2+y^2+z^2$, then the image $\pi(C_A)$ is the conic defined by $x^2+y^2+z^2$ in $\P^2$.
We call $\pi(C_A)$ the {\em elliptic absolute} and denote it by $C_E$.
It defines a metric on the real points of $\P^2$ as follows. If $p,q\in\P^2$, then the line through $p,q$ intersects $C_E$ in two complex
conjugated points $a,\bar{a}$. The cross ratio $(p,q;a,\bar{a})$ is a complex number with modulus $1$. Its logarithm is $r\ci$, where $r$
is real in the interval $[-\pi,\pi)$. Then $d(p,q):= \frac{|r|}{2}$. It is straightforward to show that $d(p,q)$ is equal to the Euclidean angle
between the lines $\pi^{-1}(p)$ and $\pi^{-1}(q)$ in $\R^3$.

A three-dimensional {\em scenario} consists of geometric objects in $\P^3$ such as points, lines, curves and surfaces that may have
Euclidean features (e.g. angles).
If $S$ is an algebraic surface, then we may mathematically define the {\em picture}~$S'$ of~$S$
as the Zariski closure of the critical value set of the restriction
of the projection map to the smooth part of $S$ (which is a differential manifold).
This is an algebraic curve. If $\pi$ is the map $(x:y:z:w)\mapsto(x:y:z)$,
then the equation of $S'$ is a factor of the discriminant of the equation of $S$ with respect to $w$.
The {\em picture} of a point or curve object is its image under $\pi$.
Finally, the {\em picture} of a scenario consisting of several objects is the union of the pictures of the individual objects.

Calibration from a picture is the task of finding the equation of $C_E$ when the picture
of a scenario is given. Here is the description of the haystack in which we want to find the needle:
the equation of of the absolute conic has six coefficients, and it is unique up to multiplication by
a scalar. So the set of irreducible conics without real points corresponds to an open subset $V\subset\P^5$.

\begin{figure}[!ht]
\begin{center}
\includegraphics[height=5cm]{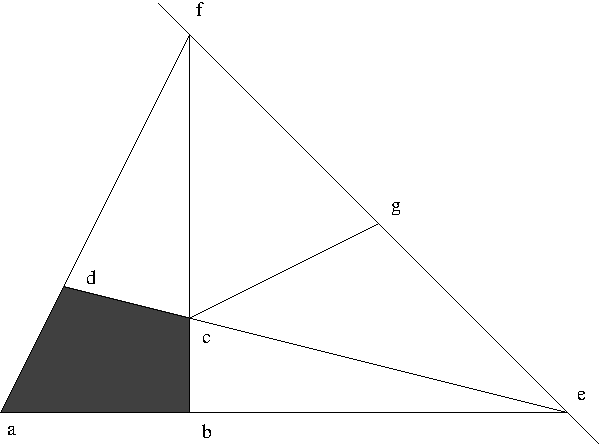}
\caption{The square $abcd$ has two pairs of parallel lines, namely ($L_{ab},L_{cd})$ and $(L_{bc},L_{ad})$.
	They intersect at the infinite points $e$ and $f$. The line $L_{ef}$ is the infinite line of the plane in which the square lies.
	The diagonal $L_{ac}$ intersects the infinite lines in a third infinite point $g$.
	The two most important points $z$ and $\bar{z}$ also lie on the infinite line; they cannot be shown because they are not real.
	These complex points lie on the Euclidean absolute, and their pictures can be used for calibration.}
\label{fig:sq}
\end{center}
\end{figure}

\begin{example}[Squares]
Let $a,b,c,d\in\P^3$ be four coplanar points, not on the infinite plane, forming a square.
The lines $L_{ab}$ and $L_{cd}$ are coplanar and therefore intersect in a point, say $e$.
Likewise, the lines $L_{bc}$ and $L_{ad}$ intersect in a point $f$. Then $L:=L_{ef}$ is the intersection of the plane $E_{abcd}$ with the infinite plane.
The diagonal $L_{ac}$ intersects $L$ in another point, say $g$. The configuration of these points is shown in Figure~\ref{fig:sq}.
The line $L$ intersects $C_A$ in two conjugate complex points $z,\bar{z}$ that cannot be shown in the figure.
Let $a',b',c',d',e',f',g',\bar{z},\bar{\bar{z}}\in\P^2$ be the images of all these points. Given $a',b',c',d'$, we can determine
\[ e' := L_{a'b'}\cap L_{c'd'}, \ f' := L_{b'c'}\cap L_{a'd'}, L' := L_{e'f'}, \ g' := L'\cap L_{a'c'} . \]
For the viewing angles of the points at infinity, we have $\angle(e,g)=\angle(g,f)=\pi/4$. This determines all cross ratios of the five points
$e,f,g,z,\bar{z}\in L$. These cross ratios are equal to the corresponding cross ratios $e',f',g',\bar{z},\bar{z}'\in L'$. Since we already know $e',f',g'$,
we can determine $\bar{z}$ and $\bar{z}'$. We have obtained two points on the elliptic absolute $C_E$. These give rise to two linear equations for the
coefficients of $C_E$.
The pictures of three squares allows for unique determination of the equation of $C_E$ (see \cite[Example~8.18]{HartleyZisserman}).
It should be pointed out that the squares
should not lie in parallel planes, otherwise we get the same points $\bar{z},\bar{z}'\in C_E$ twice.
\end{example}

The construction of two complex points on the elliptic absolute from a given square picture can be generalised to other planar scenarios.
We just need two parallel lines and two angles. In particular, any regular $n$-gon for $n\ge 4$ can be used.

\section{Surfaces of Revolution}
\label{sec:2}

This section gives a short treatment of a method for calibration using pictures of surfaces of revolution,
which has been introduced in \cite{WMC}.

Any one-dimensional group of rotations with a fixed axis is called a {\em revolution}.
A real algebraic surface is a {\em surface of revolution} if it is invariant
under some revolution. The axis is unique unless the surface is a sphere -- in that
case, any line through the center is an axis of revolution.

In the projective plane, we say that a projective automorphism $s:\P^2\to\P^2$ is
a {\em projective reflection} iff it has order two.

\begin{proposition} \label{prop:sym}
Any projective reflection $s:\P^2\to\P^2$ has a fixed line and an isolated fixed point.
If we choose projective coordinates such that the fixed line has equation $z=0$ and the fixed point
is $(0:0:1)$, then $s$ is represented by the diagonal matrix with entries $-1,-1,1$.
\end{proposition}

\begin{proof}
See \cite[Remark 1.1.19]{PW}.
\end{proof}

In the elliptic plane, we say that a projective reflection $s:\P^2\to\P^2$ is
an {\em elliptic reflection} iff it also preserves $C_E$ (and consequently,
it preserves elliptic distance). Notice that any two elliptic reflections are conjugated.
This is in contrast to the Euclidean plane, where we have two types of automorphisms
of order two, namely reflections at a line and point reflections.

\begin{proposition} \label{prop:sr}
Let $R\in\P^3$ be a surface of revolution and let $R'\subset\P^2$ be its picture.
Then $R'$ is invariant under some elliptic reflection.
\end{proposition}

\begin{proof}
This is a result of \cite[Subsection~3.3]{WMC}. Here is a short proof.
Any surface of revolution is invariant under any reflection at any plane passing
through the axis of revolution. Let $r:\P^3\to\P^3$ be the reflection at the plane
passing through the axis $L$ of revolution that fixes $R$ and through $o$, the
center of the projection $\pi:\P^3\to\P^2$. Let $L':=\pi(L)$, and let $r':\P^2\to\P^2$ be the elliptic
reflection with reflection line $L'$. Then $r'$ maps $R'$ to itself.
\end{proof}

Recall that the picture of an algebraic surface is an algebraic curve.
So, if the algebraic surface is a surface of revolution, then it should be helpful if we can compute
the projective reflections preserving a given algebraic curve $C$.

Let $F\in\R[x,y,z]$ be a homogenous polynomial and let $A\in\R^{3\times 3}$ be a nonsingular matrix.
We write $F\circ\phi_A$ for the polynomial obtained by substituting the three entries of $\phi_A(x,y,z)$ into $F$,
where $\phi_A$ is the linear map represented by the matrix $A$.

We are mainly interested in algebraic curves that have exactly one projective reflection symmetry.
If this is the case, then the unique symmetry is computed by the \Cref{alg}.

\renewcommand{\thealgorithm}{FindSymmetry} 
\begin{algorithm}[!ht]
\caption{}\label{alg}
\begin{algorithmic}[1]
  \Require $F\in\R[x,y,z]$, the defining polynomial form of an algebraic curve $C\subset\P^2$.
  \Ensure the unique projective reflection that preserves $C$ represented by its matrix in $\R^{3\times 3}$.
	If there is no such projective reflection, or if there are several, then return an error message.
  \Statex
  \State Let $A$ be a $3\times 3$ matrix with undetermined coefficients.
  \State Let $E$ be the set of all coefficients of $F-F\circ\phi_A$ with respect to $x,y,z$.
  \State Let $D$ be the set of entries of the matrix $A^2-I_2$ together with $\det(A)-1$.
  \State compute the zero set $Z$ of $E\cup D$
  \If{$|Z|=1$}
    \State instantiate $A$ by the unique solution and \Return $A$.
  \Else
    \State \Return{``not uniquely solvable''}
  \EndIf
\end{algorithmic}
\end{algorithm}

\begin{remark}
Using invariant theory, we can construct curves with any given finite number of projective
reflection symmetries. Figure~\ref{fig:sym}, taken from a workshop with high school students, shows algebraic curves
of degree 10 and 12 with 10 such symmetries. However, we believe that it is rare that the picture of a surface of revolution
has more than one symmetry axis, unless the surface is a sphere.
\begin{figure}[!hb]
\begin{center}
\includegraphics[height=3cm]{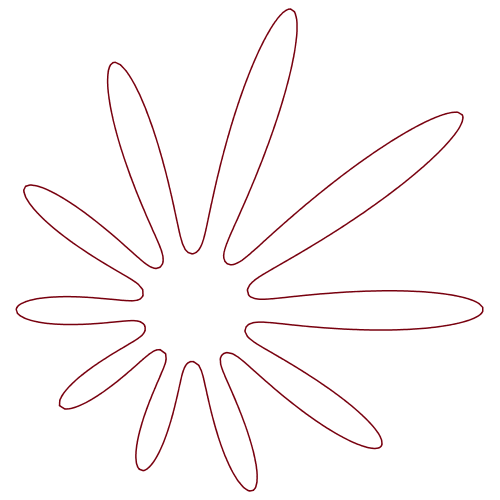}\hspace{3cm}
\includegraphics[height=3cm]{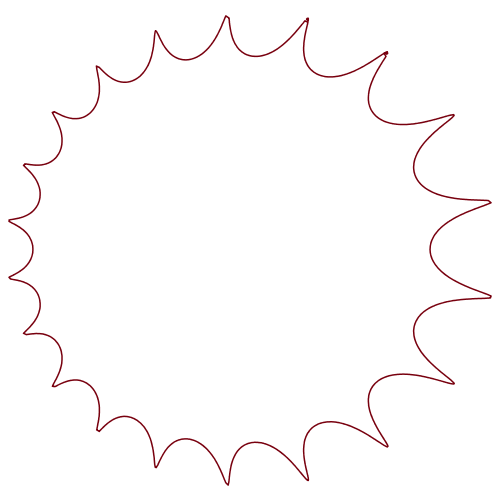}
\caption{Algebraic curves with 10 projective symmetry lines. The equations are $G-F^6/3+1/80=0$ (left) and $G^2-(F-1)^3-10+F^5=0$ (right),
	where $F=x^2+y^2$ and $G=(x+iy)^{10}+(x-iy)^{10}$. Despite the formula given for $G$, it is a polynomial with real coefficients,
	the imaginary parts cancel each other. For both curves, we apply a projective transformation to the equations, so that
	the symmetries are not Euclidean.}
\label{fig:sym}
\end{center}
\end{figure}

If the surface is a sphere, then its picture is a conic. The group of projective automorphisms of an irreducible conic $C$
is three-dimensional. (One way to see this is by looking at $C_E$: its automorphism
is the group preserving elliptic distance, which is indeed three-dimensional.) The subset of projective reflections preserving $C$
is infinite. Even worse: if $C$ is not equal to $C_E$, most projective reflections preserving $C$ do not
preserve $C_E$ and therefore they are not elliptic reflections. So, we do not get a constraint for $C_E$ from
the knowledge of projective reflections of the picture of a sphere.

However, two spheres can be considered as a single surface of revolution. In \cite{ZZW},
this observation is used to reduce the problem of calibration with given three pictures of spheres to the problem of calibration
with three pictures of surfaces of revolution.
\end{remark}

\begin{remark}
If $S$ is a semi-algebraic surface with boundary consisting of algebraic curves, then its picture is the union of
the critical value set of $\pi|_S$ and the images of the boundary curves.
In particular, if the surface of revolution is a cylinder, then its picture consists of two line segments and two conics. Both lines are
tangential to both ellipses. The four points of tangency are four distinguished points, and the action of the elliptic reflection
preserving the picture permutes these four points: any tangential intersection of an ellipse and a line is mapped to the tangential
intersection of the same ellipse with the other line. A projective transformation is uniquely determined by four
preimage/image point pairs, assuming that no three image points are collinear. Computationally, any preimage/image point pair gives rise
to two linear conditions on the eight parameters of a projective transformation. So, the computation of the unique projective
reflection that preserves the picture of a cylinder can be reduced to solving a system of linear equations. The same is true
for the picture of a truncated cone.
\end{remark}

Let us assume that we have found a projective reflection $s:\P^2\to\P^2$ preserving the picture $R'$ of some surface $R$ of revolution.
By \Cref{prop:sr}, $s$ is an elliptic reflection.
The following proposition shows that $s$ can be used to find
linear constraints for the equation vector of $C_E$.

\begin{proposition}
For any projective reflection $s:\P^2\to\P^2$, there is a four-dimensional space of quadratic forms
containing the defining forms of
all irreducible conics that are preserved by $s$.
\end{proposition}

\begin{proof}
Without loss of generality, we may assume that $s$ is the map $(x:y:z)\mapsto(-x:-y:z)$.
Then the set of conics that are invariant under $s$
is the set of conics with an equation of the form $ax^2+bxy+cy^2+dz^2=0$, with $a,b,c,d\in\R$.
\end{proof}

If we have the pictures of three surfaces of revolution, then we get enough linear equations to compute $C_E$, assuming that the
linear equations are linear independent.

\section{The Torus}
\label{sec:3}

Recall that the points in dual projective space $(\P^3)^\ast$ are in bijective correspondence with the planes in $\P^3$. The dual of an algebraic
surface $S$ is the Zariski closure of the subset of $(\P^3)^\ast$ consisting of points corresponding to tangent planes to $S$ at nonsingular points.
If $H\subset\P^3$ is a plane, and $C\subset H$ is a planar curve, then the dual $C^\ast$ is defined as the Zariski closure of the set of points
corresponding to planes that contain a tangent line to $C$ at a nonsingular points. This is a singular cone in $(\P^3)^\ast$ and its vertex is the
point $H^\ast$ corresponding to the plane $H$. In particular, the dual $C_A^\ast$ of the Euclidean absolute is a singular quadric cone.

A torus is a quartic surface in $T\subset\P^3$ with an equation
\[ (x^2+y^2+z^2+aw^2)^2-b(x^2+y^2)w^2=0 \]
in suitable coordinates, where $b>0$ and $a>0$ are real parameters.

With dual coordinates $\bar{x},\bar{y},\bar{z},\bar{w}$, the equation for the dual $T^\ast$ of the torus becomes
\[ 16(a\bar{x}^2+a\bar{y}^2+a\bar{z}^2+\bar{w}^2)^2-8b(\bar{x}^2+\bar{y}^2)(a\bar{z}^2+2\bar{w}^2)+(b^2-8ab)\bar{z}^4-8b\bar{z}^2\bar{w}^2 = 0, \]
so $T^\ast$ is also a quartic surface. It is well-known (\cite{young}) that $T^\ast$ is projectively isomorphic to $T$; it would be
interesting to find an explicit isomorphism, but we have to leave it as an unsolved exercise.

\begin{figure}
\begin{center}
\includegraphics[height=3cm]{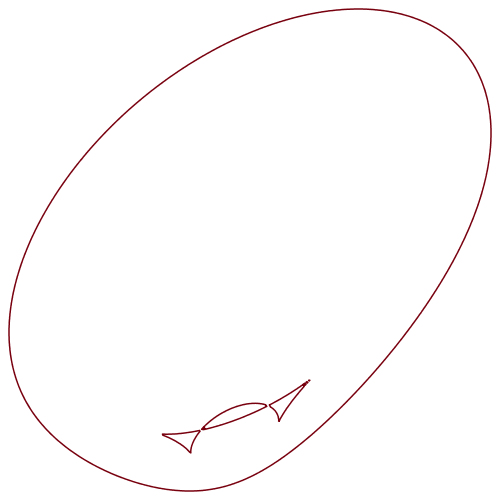}\hspace{3cm}
\includegraphics[height=3cm]{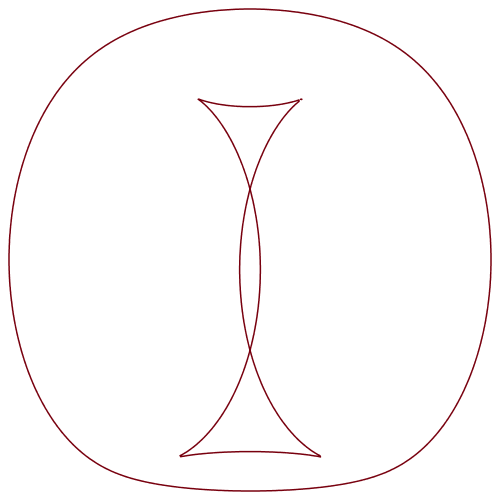}
\caption{The picture of a torus is a curve of degree~8 that may have two real components. Here are two such pictures. Both curves
have two real nodes and four real cusps. They also have six nonreal nodes and eight nonreal cusps; the nonreal nodes play a role in
the calibration algorithm.}
\label{fig:torus}
\end{center}
\end{figure}

The picture $T'$ of the torus (see Figure~\ref{fig:torus}) is a subset of the projective plane whose points correspond to lines through the center $o$.
Lines in the projective plane correspond to planes through~$o$. Tangent lines to~$T'$ correspond to tangent planes to $T$ through $o$.
This implies that the dual~$T'^\ast$ of~$T'$ is the intersection of the dual~$T^\ast$ of the torus with the plane~$o^\ast\subset(\P^2)^\ast$
(see also \cite{Jan}).
We call $T'^\ast$ the {\em dual picture} of the torus.

Notice also that the singular quadric cone $C_A^\ast$ intersects the plane $o^\ast$ in a conic. This conic is dual to the conic $C_E$ in the image plane
of the central projection. We denote it by $C_E^\ast$ and call it the {\em dual elliptic absolute}.

\begin{lemma} \label{lem:dual}
The dual picture of a torus from a generic projection center is a plane quartic curve
that intersects the dual elliptic absolute
tangentially in all its intersection points.
\end{lemma}

\begin{proof}
Straightforward computation shows that the two surfaces $T^\ast$ and $C_A^\ast$ intersect
in the quartic curve $Q$ defined by $\bar{x}^2+\bar{y}^2+\bar{z}^2$ and $b\bar{z}^2+4\bar{w}^2$, with intersection multiplicity two.
The dual picture is the intersection of $T^\ast$ with a generic plane. 
The intersection points of the dual picture and $C_E^\ast$ are the intersections of $Q$ with the same generic plane,
and they are therefore also double intersections.
\end{proof}

\begin{lemma} \label{lem:nodes}
The picture of a torus from a generic projection center is a plane curve of degree eight. It has two nodes lying on the elliptic absolute.
\end{lemma}

\begin{proof}
The degree of the picture of any Darboux cyclide with a nodal double curve is eight, by \cite{HMS}. The torus has $C_A$ as its nodal curve, which shows
the first claim.

The torus is the stereographic projection of a \emph{ring cyclide},
namely a quartic surface in the projective sphere in $\P^4$
that has two pairs of complex conjugate isolated nodal singularities
and is covered by four pencils of circles (see \citep[Example~16]{2023d}).
Moreover, two of these pencils each have two complex conjugate base
points at these singularities.
The conics in one of these pencils are projected to a pencil of circles
on the torus in $\P^3$ lying in parallel planes. Thus,
these parallel planes meet at a line at infinity. This line
meets the Euclidean absolute in two points.
Since the intersection of the torus with the plane at infinity is the
Euclidean absolute,
we deduce that these two points correspond to base points of the pencil.
It follows that the stereographic projection maps two of the singular
points
onto the Euclidean absolute.

We claim that the central projection maps the nodal point of the torus to a nodal
point of its picture. To see this, think of the vertex of a singular quaric cone.
(It suffices to analyse just one case, because any two nodal points are analytically
isomorphic.) The picture of the cone consists of two lines, intersecting in the
image of the vertex.

As the nodal point of the torus lies on the Euclidean absolute, the nodal point
of the picture lies on the elliptic absolute.
\end{proof}

From now on, we assume that both the image $T'$ and the dual image $T'^\ast$ are known. This is clearly redundant, because any curve
determines its dual, hence $T'$ determines $T'^\ast$ and conversely $T'^\ast$ determines $T'$.
However, computing the dual of a plane curve may be expensive.
Moreover, the starting point in applications is an image of a curve. Such an image allows to obtain, through measurement,
points on the curve but also tangents to the curve; and these are the basis for deriving the defining forms for $T'$ and for $T'^\ast$.
We mention that the task of finding the defining forms is numerically non-trivial (see \cite{WMC} for an analysis of the effects of
numerical noise).
Here we concentrate on the task of finding the defining form for $C_E$. This is equivalent to finding the form for
$C_E^\ast$: the symmetric matrix associated to the quadratic form defining $C_E^\ast$ is the inverse of the symmetric matrix associated
to the quadratic form defining $C_E$, up to scalar multiplication.

\Cref{lem:nodes,lem:dual} give constraints for the elliptic absolute and for the dual elliptic absolute.
More precisely, \Cref{lem:nodes} gives rise to two linear conditions on the coefficients of the quadratic form defining $C_E$,
and \Cref{lem:dual} gives four nonlinear conditions on the coefficients of the quadratic form defining $C_E^\ast$. In addition,
the torus is also a surface of revolution and its picture has a reflection symmetry. As explained in \Cref{sec:2},
the projective reflection symmetry~$s$ of~$T'$ also
gives rise to two linear conditions on the coefficients of the quadratic form defining $C_E$. Equivalently, we can get
to two linear conditions on the coefficients of the quadratic form defining $C_E^\ast$, using the dual projective reflection~$s^\ast$,
which preserves $T'^\ast$ and whose matrix representation is the transpose of the matric representation of~$s$.

\begin{remark}
The generic picture of a torus is a quartic of genus one with 8 nodes and 12 cusps (possibly complex). Any projective symmetry must
send nodes to nodes and cusps to cusps. By checking
random examples, it can be checked that a generic genus one quartic does not have any projective symmetry, and that a generic picture
of a torus has just a single projective symmetry. On the other hand, the picture of a torus by parallel projection has
three projective symmetries.
\end{remark}

The constraints above are not independent. The two nodes in \Cref{lem:nodes} are related by the symmetry $s$.
Any conic that is invariant under $s$ and that passes through one node also passes through the second node. Similarily,
any conic that is invariant under $s'$ that has a tangential intersection point with $T'^\ast$ also has a second tangential intersection
point symmetric to it.

In order to find if a potential candidate (maybe with unknown symbolic coefficients) intersects a given curve tangentially everywhere,
we compute the resultant of the defining forms, and test if the resultant is a perfect square. Finding the subset of polynomials
of degree~$2k$ that are perfect squares is done by comparing the polynomial with the square of a general polynomial of degree~$k$
with unknown coefficients. We obtain $2k+1$ equations in~$k+1$ additional variables. The additional variables can easily
by eliminated by the available elimination algorithm in Maple. We also get inequality conditions: the Hessian of the candidate
has to be nonzero, otherwise the elliptic absolute would not be irreducible; and the candidate should not pass through a singular
point of~$T'^\ast$, otherwise some of the tangential intersections would degenerate.

The method is summarized by \Cref{alg2}.

\renewcommand{\thealgorithm}{TorusCalibration} 
\begin{algorithm}[ht]
\caption{}\label{alg2}
\begin{algorithmic}[1]
  \Require a torus picture $T'\subset\P^2$ from a generic point and its dual $T'^\ast\in(\P^2)^\ast$.
  \Ensure a finite list of candidates that include the dual absolute conic $C_E^\ast$.
  \Statex
  \State Compute the projective reflection $s^\ast:(\P^2)^\ast\to (\P^2)^\ast$ that preserves $T'^\ast$, using \Cref{alg}.
  \State Compute the vector space $V$ of quadratic forms that are preserved by $s^\ast$ ($\dim(V)=4$).
  \State Select a non-real node in $n\in T'$ that is not fixed by $s$ (the dual of $s^\ast$).
  \State Compute the set $S_1$ of quadratic forms $f\in V$ such that the line $n^\ast$ is tangential to the zero set of $f$ (three-dimensional).
  \State Compute the set $S_2$ of quadratic nondegenerate forms $g\in V$ such that $T'^\ast$ intersects the zero set of $f$
	tangentially in every point (two-dimensional).
  \State Solve the union of the defining equations for $S_1$ and $S_2$ and \Return{all solutions} (finitely many up to scalar multiplication).
\end{algorithmic}
\end{algorithm}

\begin{example} \label{ex:fin}
The genus one quartic $T'^\ast$ with equation
\begin{gather*}
\bar{x}^{4}+\left(-\tfrac{81}{32} \bar{y}^{2}-\tfrac{32}{81} \bar{y} \bar{z} -\tfrac{8}{9} \bar{z}^{2}\right) \bar{x}^{2}+\tfrac{6561 \bar{y}^{4}}{4096}+\tfrac{\bar{y}^{3} \bar{z}}{2}-\tfrac{10 \bar{y}^{2} \bar{z}^{2}}{9}
\\
-\tfrac{1755136 \bar{y} \,\bar{z}^{3}}{77058945}+\tfrac{3008303104 \bar{z}^{4}}{905057309025} = 0
\end{gather*}
has a unique projective symmetry $s:(\bar{x}:\bar{y}:\bar{z})\mapsto (-\bar{x}:\bar{y}:\bar{z})$ and 8 bitangents.
In particular, the two complex lines $\bar{x}\pm\ci\bar{y}=0$ are bitangents. They correspond to two complex nodes of the dual curve $T'$, with coordinates
$(x:y:z)=(1:\pm\ci:0)$. This quartic has been constructed by setting up a quartic with symbolic coefficients that has symmetry $s$,
and then solving a system of equation imposed by the constraint that
$\bar{x}\pm\ci\bar{y}=0$ is a bitangent and by the constraint that $T'^\ast$ has genus one.

The general quadratic form in $\bar{x},\bar{y},\bar{z}$ that is invariant under $s^\ast$ is of the form
$a_1\bar{x}^2+a_2\bar{y}^2+a_3\bar{y}\bar{z}+a_4\bar{z}^2$. The lines $\bar{x}\pm\ci\bar{y}=0$ are tangents if and only if
$4a_1a_4-4a_2a_4+a_3^2=0$; this equation is the discriminant of the univariate polynomial obtained by plugged a parametrization of one of
the lines into the quadratic form (both lines give the same equation).
The system of equations arising from the
tangential intersection constraint is too big to be written here.
We refer to the Maple script available at \cite{2024code}.
The Hessian is $2a_1(4a_2a_4-a_3^2)$.
The full system of equations and inequalities has exactly one solution. We instantiate the general
quadratic form by this one solution and compute the dual form
\begin{gather*}
80478208 x^{2}+80478208 y^{2}-3692252160 y z +99394940025 z^{2} .
\end{gather*}
The paper \cite{HMS} (see also \cite{GLSV}) contains an algorithm that computes the defining form of a Darboux cyclide when its picture and the elliptic absolute
are known. Because tori are special cases of Darboux cyclides, we can apply that algorithm to $T'$ and get the equation
of the torus (which, again, is too big to be written here).
\end{example}

\section{Conclusion}

This paper leaves two questions open. First, the given algorithm requires exact data, and these are not available when the input is coming
from measurements on a picture. So, an algorithm that is useful in practice must be numerical. This leads to the question how sensitive
the computed output is to small numerical perturbations.

Another question has been raised by a reviewer: can we say something about the number of solutions? The solution in Example~\ref{ex:fin}
was computed using a particular choice of two conjugated nodes in $T'$. In this example, there are six complex nodes outside the
symmetry line. This implies that there are two other choices of a pair of complex conjugate nodes outside the symmetry plane.
So there could be two more solutions.

\section*{Acknowledgements}

Financial support for Niels Lubbes was provided by the Austrian Science Fund (FWF): P36689.
Josef Schicho was supported by the European Union's Horizon 2020 programme in the frame of the MSCN 860843 (GRAPES) project.

\textbf{Addresses of authors:}
\\Institute for Algebra, Johannes Kepler University, Linz, Austria
\\\url{info@nielslubbes.com}
\\Research Institute for Symbolic Computation (RISC),
Johannes Kepler University, Linz, Austria
\\\url{josef.schicho@risc.jku.at}

\end{document}